\documentclass[a4paper,11pt,reqno]{amsart}

\usepackage{array}

\usepackage[latin1]{inputenc}

\def\theoremnumberstyle{changesection} 

\usepackage{michael}

\theoremstyle{break-italic}
\swapnumbers

\numberwithin{equation}{section} 
\allowdisplaybreaks[3]   

\title{Gorenstein Modules of Finite Length}
\author{Michael Kunte}
\address{Fachgebiet Mathematik, Universität des Saarlandes, 66041 Saarbrücken}
\email{kunte@math.uni-sb.de}
\begin{document}

\maketitle

\begin{abstract}
In Commutative Algebra structure results on minimal free resolutions of Gorenstein modules are of classical interest. We define Gorenstein modules of finite length over the weighted polynomial ring via symmetric matrices in divided powers. 
\\
We show that their graded minimal free resolution is selfdual in a strong sense. Applications include a proof of the dependence of the monoid of Betti tables of Cohen-Macaulay modules on the characteristic of the base field. Moreover we give a new proof of the failure of the generalization of Green's Conjecture to characteristic $2$ in the case of general curves of genus $2^n -1$.   
\end{abstract}

\section{Introduction}

The Theorem of Buchsbaum and Eisenbud \cite[Theorem 2.1]{BE77} states that every Gorenstein ideal of $\depth=3$ has a skew symmetric resolution. It is easy to build ideals with skew symmetric resolutions in a polynomial ring with three variables: By Macaulay's Theorem \cite{M16} every such ideal occurs as the annihilator of a homogeneous form in divided powers. Here the polynomial ring acts on the divided power algebra. 
\\
We consider finite length modules over weighted polynomial rings with arbitrary many variables. All these are realized as quotients of annihilators of matrices in divided powers. The question is: What are sufficient conditions for these matrices to gain a module with a symmetric respectively skew symmetric resolution? Additionally: Are these conditions necessary?       
\\
The main result of this paper is the answer to both of the two questions in the case of an odd number of variables. 

\medskip

Another motivation to study modules of finite length with a skew symmetric resolution over the usual polynomial ring $R=k[x_1,\ldots,x_n]$ is the following: In their recent paper Eisenbud and Schreyer \cite{ES08} prove a strengthened form of the Boji-S\"oderberg conjectures \cite{BS06}. In this context the monoid of Betti tables of minimal free resolutions of Cohen-Macaulay $R$-modules is studied. The group structure is given by pointwise addition as the Betti table of the direct sum of two modules is the sum of the Betti tables. Eisenbud and Schreyer conjecture that the monoid depends on the characteristic of the ground field $k$. We can apply our theory and symmetric resolution constructions to prove this conjecture by focusing on a certain example.

\subsection[The Main Result]{The Main Result}
The main result is developed in Sections \ref{sec3} and \ref{sec5}. The major aspects can be summarized briefly as follows: 
\\
Let $R=k[x_1,\ldots,x_n]$ be the weighted polynomial ring over an arbitrary field $k$ with $\deg x_l = d_l > 0$. Let $M$ be a graded $R$-module of finite length. For us $M$ is said to be Gorenstein if for some $s \in \mathbb{Z}$ there is a graded isomorphism $\tau : M \rightarrow \Hom_k(M,k)(-s)$ such that $\tau = \Hom_k(\tau,k)(-s)$. A slightly more general definition for Gorenstein modules of finite length is presented in \ref{gor-def}. Another way to put it, which is especially useful from the computational point of view, is the following:

\begin{theorem}
\label{theo1}
$M$ is Gorenstein if and only if there is a symmetric matrix $P$ in divided powers, such that we have
\[
M \cong M(P).
\]
\end{theorem}
Here $M(P) =  \bigoplus_{j=1}^p R(b_j) / \Ann_R(P)$ is the quotient of the annihilator of $P$. The exact definition of $\Ann_R(P)$ is given in \ref{defann}. The proof of Theorem \ref{theo1} follows from Theorem \ref{Gorensteintheorem}. 

\medskip

For Gorenstein modules the following theorem holds over any characteristic:

\begin{theorem}[Selfdual Resolution] \label{theo2.0}
Let $n \geq 3$ be an odd integer, and let $m = \frac{n-1}{2}$. Assume that $M$ is Gorenstein with $\tau = \Hom_k(\tau,k)(-s)$, and let $(\quad)^{\vee} = \Hom_R(\quad, R(-\sum_{l=1}^n d_l -s))$. Then there is a graded free resolution of $M$ of the form 
\[
0 \leftarrow M \leftarrow F_0 \stackrel{\phi_1}{\leftarrow} F_1 \leftarrow \ldots \leftarrow F_{m}\stackrel{\phi_{m+1}}{\leftarrow}  (F_{m})^{\vee} \leftarrow \ldots \leftarrow (F_1)^{\vee} \stackrel{\phi_1^{\vee}}{\leftarrow} (F_0)^{\vee} \leftarrow 0,
\]
such that $\phi_{m+1}$ is skew if $m$ is odd and symmetric if $m$ is even.  
\end{theorem} 

The proof is given in Theorem \ref{maintheorem}. Theorem \ref{theo2.0} is of special interest in the context of the applications. 

\medskip

In Corollary \ref{minres} we give a symmetric minimization process such that the symmetry of the resolution is kept. To do so we need $\car k \neq 2$. 
\\
In Section \ref{sec5} an equivalence is given: A symmetric resolution implies already the Gorenstein property. These two facts might be summarized as follows: 

\begin{theorem}[Selfdual Minimal Resolution] \label{theo2}
Let $\car k \neq 2$. Let $s$ be the top degree of M and $(\quad)^{\vee}= \Hom_R(\quad,R(-\sum_{l=1}^n d_l -s))$. 
Let $n \ge 3$ be an odd integer, and set $m := \frac{n-1}{2}$. Then $M$ is Gorenstein if and only if its minimal graded free resolution is of the form 
\[
0 \leftarrow M \leftarrow F_0 \stackrel{\psi_1}{\leftarrow} F_1 \leftarrow \ldots \leftarrow F_{m}\stackrel{\psi_{m+1}}{\leftarrow}  (F_{m})^{\vee} \leftarrow \ldots \leftarrow (F_1)^{\vee} \stackrel{\psi_1^{\vee}}{\leftarrow} (F_0)^{\vee} \leftarrow 0,
\] 
and satisfies the following condition: $\psi_{m+1}$ is skew if $m$ is odd and symmetric if $m$ is even. 
\end{theorem}

In Theorem \ref{backwardstheorem} the Gorensteiness of $M$ is shown to be true. The proof is derived from the more general Theorem \ref{functortheorem}, which is formulated in the language of isomorphisms of functors. 

\subsection{The Structure of this Paper}
In Section \ref{sec1} we show how to define graded modules of finite length via matrices in divided powers. The following section connects symmetric divided power matrices and Gorenstein modules. In the fourth section we prove the strong selfduality of the resolution in the Gorenstein case using resolution constructions due to Nielsen. Section \ref{sec5} is devoted to the applications for the monoid of Betti tables and Green's Conjecture. Finally the last section provides a proof of the fact that a module of finite length with a symmetric resolution is Gorenstein. 

\subsection{Acknowledgments}
I thank especially Frank-Olaf Schreyer for his suggestions and efforts throughout this project. 

\section{Notation and the Module $M(P)$}
\label{sec1}
Let $k$ be a field of arbitrary characteristic. Let $R = k[x_{1},\ldots,x_{n}] = \bigoplus_{j \geq 0} R_{j}$ be the weighted polynomial ring, with weights $\deg x_i = d_i > 0$. Let $d = \sum_{\ell =1}^n d_{\ell}$.  
Let $\kd$ be the graded dual of $R$, i.e.
\[ \kd = \grHom_k(R,k) = \bigoplus_{j \geq 0} \Hom_{k}(R_j,k) = \bigoplus_{j\geq 0} \kd_{-j}\stackrel{pr}{\longrightarrow} \kd_0 = k.\]
We denote by $x^{u}=x_1^{u_1}\cdots x_n^{u_n}, \|u\|=u_1+\cdots+u_n =j$, the standard monomial basis of $R_j$. By $X_1 = X_1^{(1)},\ldots,X_n =X_n^{(1)}$ we mean the basis dual to $x_1,\ldots,x_n$ (here $V:=\langle x_1,\ldots,x_n \rangle$ can be viewed as sub $k$-vectorspace of $R$), by $X^{(u)}=X_1^{(u_1)}\cdots X_n^{(u_n)}$ the basis elements of $\kd_{-j}$ dual to the basis $\{x^u : \|u\|=j\}$. The $X^{(u)}$ are called {\it divided power monomials}, their linear combinations {\it divided power polynomials} or simply {\it divided powers}.

\begin{definition} \label{contraction}
For all $(i,j) \in \mathbb{Z}^2$ we define natural contraction maps $ R_i \times \kd_{-j} \longrightarrow \kd_{-j+i} $ as follows: Let $\phi \in R_i$ and $f \in \kd_{-j}$. 
\[
(\phi, f) := \phi \cdot f := \left \{ \begin{array}{ll}  0 & \mbox{ for } j < i \\ R_{j-i} \rightarrow k: \psi \mapsto f(\phi \psi) & else. \end{array}  \right.  
\]

\end{definition}
\begin{remark}
$\kd$ has a graded $R$-module structure: $R$ acting on $\kd$ via the above contraction action. The induced map $R_i \times \kd_i \rightarrow \kd_0 \cong k$ is a perfect pairing. All elements of $\kd$ have negative degree, e.g. $X_1^{(2)}X_2^{(1)}$ has degree $-3$. 
\\
In terms of basis one has \[x_1^{u_1} \cdots x_n^{u_n} \cdot X_1^{(j_1)}\cdots X_n^{(j_n)} = X_1^{(j_1 - u_1)}\cdots X_n^{(j_n-u_n)}.\]
\end{remark}

\begin{definition} \index{Contraction action}
Let $b_j$, $j = 1,\ldots,p$, be integers and let $\kf = \bigoplus_{j=1}^{p} \kd(b_j)$ be the direct sum of $p$ shifted $R$-modules $\kd$. Let $G = \bigoplus_{j=1}^{p} R(-b_j)$. Let $\phi \in G$ and $f \in \kf$ then
\[
  \langle f, \phi \rangle:= \sum_{j=1}^p \phi_j \cdot f_j.
\]
Recall that $\phi_j \cdot f_j$ means the contraction action defined in \ref{contraction}. 
\end{definition}
 
 Our first aim is to define quotients of annihilators of homogeneous matrices over $\kd$:
\\
Let the $a_i$ and $b_j$ be non-negative integers and let 
\[P \in \Hom_{R}\left(\bigoplus_{j=1}^{p}R(b_j),\bigoplus_{i=1}^{q}\kd(a_i)\right)\]
 a homogeneous homomorphism, represented by a homogeneous matrix in divided powers with entries such that $\deg P_{i,j} = -b_j + a_i$. 

\begin{definition}
\label{defann} 
We define the \textit{annihilator of $P$ in $R$} to be
\[
 \Ann_R(P) := \left\{ \phi \in \bigoplus_{j=1}^{p} R(b_j) \; \bigg | \; \langle P_i, \phi \rangle = 0 \mbox{ for all } i \right \}
\]
where the $P_i$ denote the rows of $P$ considered as $P_i \in \bigoplus_{j=1}^{p}\kd(a_i-b_j)$.

\end{definition}

\begin{remark}
Clearly the annihilator is an $R$-module as $\langle P_i, r \phi \rangle = r  \langle P_i, \phi \rangle$ for all $r \in R$. Moreover $\Ann_R(P)$ is up to isomorphism independent from the matrix representation of $P$.
\end{remark}

Any graded $R$-module $M$ of finite length can be defined by a matrix in divided powers. Assume that $M$ has $p$ generators of degrees $b_j$:
\[
\bigoplus_{j=1}^p R(b_j) \stackrel{\alpha}{\rightarrow} M \rightarrow 0.
\]
As $M$ is of finite length, $\Hom_k(M,k)$ is of finite length, too. Let $\Hom_k(M,k)$ have $q$ generators of degrees $a_i$:  
\[
\bigoplus^q_{i=1}R(-a_i) \stackrel{\beta}{\longrightarrow} \Hom_k(M,k) \rightarrow 0. 
\]
We apply $(\quad)^*=\grHom_k(\quad,k)$ to the last row and obtain a diagram
\medskip
\small

\xymatrix{&&&&& 0 & 0 \ar[d]& \\ 
          &&&&& M \ar[u] \ar[r]^{\id} & M \ar[d]_{\beta^{*}}&(1) \\
	 &&&&& \oplus_{j=1}^p R(b_j) \ar[u]^{\alpha} \ar[r]^{P} & \oplus_{i=1}^q \kd(a_i).& }
\normalsize
The matrix $P$ is defined via $ \beta^* \circ \alpha$. Hence the kernel of $P$ --- considered as a map of graded $R$-modules --- is $\Ann_R(P)$. 

\medskip

That motivates a new definition for the construction of modules of finite length. It is in the spirit of the Theorem of Macaulay (see \cite[Lemma 2.12]{IK99} or \cite{M16}).  
Our central objects are quotients of annihilators $\Ann_R(P)$:

\begin{definition}
\label{MP}
Let the quotient of $P$ in $R$ be
 \[
   M(P) := \bigoplus_{j=1}^{p} R (b_j) \bigg / \Ann_R(P).
 \]
\end{definition}

\medskip

From the above diagram $(1)$ one derives immediately:
\begin{theorem}
\label{pexists} \index{Divided powers! associated matrix in} 
Let $M$ be a graded $R$-module of finite length. Then there are integers $(a_1,\ldots,a_q)$, $(b_1,\ldots,b_p)$ and a $P \in \Hom_R(\bigoplus_{j=1}^{p} R(b_j),  \bigoplus_{i=1}^{q} \kd(a_i))$       such that 
\[
M \cong M(P)
\]
as graded $R$-modules.
\end{theorem}

Let $P^t$ denote the transposed matrix of $P$, defining an element in $\Hom_{R}(\bigoplus_{i=1}^{q}R(-a_i),\bigoplus_{j=1}^{p}\kd(-b_j))$. Then the following holds.

\begin{theorem} \label{P^t} \index{Dual module! $k$-vectorspace of $M(P)$}
There is a natural isomorphism 
\[
 M(P^t) \cong \Hom_k(M(P),k)
\]
as graded $R$-modules. 
\end{theorem}

\begin{proof}
Consider the following exact sequence:
\[ 
  0 \;\longleftarrow \;\Coker P \;\stackrel{pr}{\longleftarrow}\;\bigoplus_{i=1}^{q}\kd(a_i) \; \stackrel{P}{\longleftarrow} \; \bigoplus_{j=1}^{p}R(b_j) \; \longleftarrow \; \Ann_R(P) \; \longleftarrow \; 0
\]
Note that $\bigoplus_{j=1}^{p} R(b_j) \bigg / \Ann_R(P) = M(P).$ Using $\grHom_k(\quad,k)$ we dualize the sequence to
\[
  0 \;\longrightarrow \;\grHom_k(\Coker P,k) \;\longrightarrow\;\bigoplus_{i=1}^{q}R(-a_i) \; \stackrel{P^{*} = \Hom_k(P,k)}{\longrightarrow} \; \bigoplus_{j=1}^{p}\kd(-b_j).
\]
As $(P^*)_{\mu,\nu} = (P^t)_{\mu,\nu}$ it follows that $\grHom_k(\Coker P,k) \cong \ker P^{*} = \ker P^t = \Ann_R P^t$. That is why dualizing the upper sequence we obtain
\[
  0 \;\longrightarrow \;\Ann_R(P^t) \;\longrightarrow\;\bigoplus_{i=1}^{q}R(-a_i) \; \longrightarrow \; \grHom_k(M(P),k) \; \longrightarrow \; 0.
\]
 Because of $\grHom_k(M(P),k) = \Hom_k(M(P),k)$ as $M(P)$ is of finite length we obtain 
 \[
       \Hom_k(M(P),k) \cong M(P^t).
  \]
\end{proof}

\begin{example}
Let $R = k[x_1,x_2]$ with all weights $1$, and let
 \[ P = \left (
 \begin{array}{cc}
        X_1^{(3)} & X_1^{(1)} X_2^{(1)} + X_2^{(2)}\\
 \end{array}
 \right )  \in \Hom_R(R(3) \oplus R(2) , \kd  ).
 \]
Then $\Ann_R(P) =$
 \[
 \left \langle
 \left (
 \begin{array}{c}
        x_2 \\
	0 \\
 \end{array}
 \right ),  
 \left (
 \begin{array}{c}
        x_1^2\\
	  x_1-x_2\\
 \end{array}
 \right ), 
\left (
 \begin{array}{c}
        0\\
	x_1^2 \\
 \end{array}
 \right ) 
  \right \rangle.  
\]

As a $k$-vector space basis of  $M(P)$ is given by:

\[
 \left (
 \left (
 \begin{array}{c}
        1\\
	0 \\
 \end{array}
 \right ),
 \left (
 \begin{array}{c}
        x_1\\
	0 \\
 \end{array}
 \right ),
 \left (
 \begin{array}{c}
        x_1^2\\
	0 \\
 \end{array}
 \right ),
 \left (
 \begin{array}{c}
        x_1^3\\
	0 \\
 \end{array}
 \right ),
 \left (
 \begin{array}{c}
        0\\
	1 \\
 \end{array}
 \right ),
 \left (
 \begin{array}{c}
        0\\
	x_1 \\
 \end{array}
 \right ) \right).  
\] 

We have the graded $R$-isomorphism $\phi: \Image(P) \rightarrow M(P)$ given by 
$$
\left (
 \begin{array}{c}
        X_1^{(3)}\\
 \end{array}
 \right )
 \mapsto 
\left (
 \begin{array}{c}
        1\\
	0 \\
 \end{array}
 \right ), \mbox{   and   }
\left (
 \begin{array}{c}
        X_1^{(1)}X_2^{(1)}+X_2^{(2)}\\
 \end{array}
 \right )
 \mapsto 
\left (
 \begin{array}{c}
        0\\
	1 \\
 \end{array}
 \right ).
$$
 For $\Ann_R(P^t) \subset R$ we obtain as generators:
$
\left \langle
\left (
 \begin{array}{c}
        x_1^4\\
 \end{array}
 \right ),
\left (
 \begin{array}{c}
      x_1x_2-x_2^2\\
 \end{array}
 \right ),
\left (
 \begin{array}{c}
        x_1^2x_2\\
 \end{array}
 \right ),
\left (
 \begin{array}{c}
        x_2^3\\
 \end{array}
 \right ) \right \rangle.
$
Hence $M(P^t)$ can be represented as $k$-vector space by the following basis:
\[
\left (
\left (
 \begin{array}{c}
        1 \\
 \end{array}
 \right ),
\left (
 \begin{array}{c}
        x_1\\
 \end{array}
 \right ),
\left (
 \begin{array}{c}
        x_1^2\\
 \end{array}
 \right ),
\left (
 \begin{array}{c}
        x_1^3\\
 \end{array}
 \right ),
\left (
 \begin{array}{c}
        x_1x_2\\
 \end{array}
 \right ),
\left (
 \begin{array}{c}
        x_2\\
 \end{array}
 \right )
\right )
\]

A graded $R$-module isomorphism $\Image(P) \rightarrow \Hom_k(M(P^t),k)$ can be defined by
\[
\left (
 \begin{array}{c}
        X_1^{(3)}\\
 \end{array}
 \right )
 \longmapsto \left \{ \begin{array}{c} M(P) \rightarrow k, \\ (\mm \mapsto \left \langle  
 \begin{array}{c}
        X_1^{(3)}\\
 \end{array}
  , \mm  \rangle (0) \right ),\end{array} \right. \mbox{   and   }
 \]
\[
\left (
 \begin{array}{c}
        X_1^{(1)}X_2^{(1)}+X_2^{(2)}\\
 \end{array}
 \right )
 \longmapsto \left \{ \begin{array}{c} M(P) \rightarrow k, \\ (\mm \mapsto \left \langle 
 \begin{array}{c}
        X_1^{(1)}X_2^{(1)}+X_2^{(2)}\\
 \end{array}, \mm  \rangle (0) \right ), \end{array} \right. 
\]
 where $f(0)$ denotes for any $f \in \kd$ the projection $ \kd \stackrel{ \pr}{\rightarrow} \kd_0 = k$. 
\end{example}

\section{Gorenstein Case}

We focus now on the Gorenstein case. The main theorem within this section basically states that a Gorenstein module can be described by a symmetric matrix. 
\\
At first let us define a reasonable Gorenstein notion. Let $M$ be always a graded $R$-module of finite length. 

\begin{theorem} \label{Gorensteintheorem} 
The following two properties are equivalent: 
\begin{itemize}
\item There is a graded isomorphism  $\tau : M \longrightarrow \Hom_k(M,k)(-s)$ for some integer $s$ such that: $ \tau^{*} := \Hom_k(\tau,k) : M(s) \longrightarrow  \Hom_k(M,k)$ has the property $ \tau^{*}(-s) = \tau.$
\item There exists an $P \in \Hom_R \left (\bigoplus_{i=1}^p R(-a_i), \bigoplus_{i=1}^p \kd(a_i-s)\right)$ for some integers $p, a_1,\ldots, a_p$  such that the transposed matrix $P^t$ has the property $ P^t(-s) = P $ and $M \cong M(P)$.

\end{itemize}
\end{theorem}

\begin{proof}
First let us assume the upper property. Let $M$ be generated by $p$ generators of degrees $a_i$, then also \mbox{$0 \rightarrow   \Hom_k(M,k)(-s)  \stackrel{\alpha^{*}=\Hom_k(\alpha,k)}{\longrightarrow} \oplus_{i=1}^p \kd(a_i-s)$}. $P$ is defined via the diagram:
 
\small
\xymatrix{&&&&& 0 & 0 \ar[d] \\ 
          &&&&& M \ar[u] \ar[r]^{\tau \qquad} & \Hom_k(M,k)(-s) \ar[d]_{\alpha^{*}} \\
	 &&&&& \oplus_{i=1}^p R(-a_i) \ar[u]^{\alpha} \ar[r]^{P\quad} & \oplus_{i=1}^p \kd(a_i-s). }
\normalsize
 
 Consider the dual of the diagram:

\small
\xymatrix{& &&&& 0 \ar[d]  & 0 \\
         & &&&& \Hom_k(M,k) \ar[d]_{\alpha^{*}}  &  \ar[l]_{\tau^{*}} \ar[u]  M(s)\\
	&&&&& \oplus_{i=1}^p \kd(a_i)  & \oplus_{i=1}^p R(s-a_i) \ar[l]_{P^*} \ar[u]^{\alpha}. \\}
\normalsize 
  
Now it follows $ <P^*(e_j),e_i> = < \alpha^* \tau^* \alpha(e_j), e_i > = <\alpha^* \tau \alpha(e_j),e_i > = P_{i,j}$ by assumption. If $e_i$ denotes the $i$-th unit vector of $\bigoplus_{i=1}^p R(-a_i)$, and $e_l^*$ the $l$-th of $\bigoplus_{i=1}^p \kd(a_i)$ then $<P^*(e_j), e_i> = <e_j^{**} \circ P,e_i> =  < \sum_{l=1}^p <P(e_l),e_j > e_l^*, e_i  > =< P(e_i),e_j> = P_{j,i}$, hence $P= P^t(-s)$. By definition $M \cong M(P)$.

\medskip
Assuming the lower property let $\alpha : \bigoplus_{i=1}^p R(-a_i) \rightarrow M(P) \rightarrow 0$ be the projection, and let \mbox{$\alpha^* : 0 \rightarrow \Hom_k(M(P),k)(-s) \rightarrow \bigoplus_{i=1}^p \kd(a_i-s)$} be the dual map. 
\\
Then define $\tau'$ again by a diagram: 

\small

\xymatrix{ &&&&&0 & 0 \ar[d] \\ 
           &&&&&M(P) \ar[u] \ar[r]^{\tau' \quad } & \Hom_k(M(P),k)(-s) \ar[d]_{\alpha^{*}} \\
	 &&&&&\oplus_{i=1}^p R(-a_i) \ar[u]^{\alpha} \ar[r]^{P} & \oplus_{i=1}^p \kd(a_i-s). }
	 
\normalsize	 

Let $\mm \in M(P)$ then choose an $\tilde{\mm} \in \bigoplus_{i=1}^p R(-a_i)$
such that $\alpha(\tilde{\mm}) = \mm$. $P(\tilde{\mm})$ is well-defined, as $\alpha(\tilde{\mm}-\bar{\mm}) = 0$ if and only if $\tilde{\mm}-\bar{\mm} \in \Ann_R(P)$, hence we have $P(\tilde{\mm}-\bar{\mm}) = 0$.
Moreover we have $\langle P(\tilde{\mm}), b\rangle=0$ for all $b \in \Ann_R(P)$ as $P$ is symmetric. Hence $P(\tilde{\mm}) \in \Image \alpha^*$. 
Therefore we can define $\tau '(\mm) := (\alpha^*)^{-1}(P(\tilde{\mm}))$. By definition $\tau'$ is an isomorphism. 

As $(\tau')^*$ is defined by the diagram 

\small

\xymatrix{  &&&&&0 \ar[d]  & 0 \\
          &&&&& \Hom_k(M(P),k) \ar[d]_{\alpha^{*}}  &  \ar[l]_{(\tau')^{*}} \ar[u]  M(P)(s)\\
	 &&&&&\oplus_{i=1}^p \kd(a_i)  & \oplus_{i=1}^p R(s-a_i) \ar[l]_{P^*=P} \ar[u]^{\alpha} \\}
\normalsize
we obtain $(\tau')^*(-s) = \tau'$. If $\beta: M \rightarrow M(P)$ is the given isomorphism, define $\tau:= (\beta)^*(-s) \circ \tau' \circ \beta$. 
\end{proof}

From the above theorem and from $\Hom_k(M,k) \cong \Ext^n_R(M,R(-\sum_{l =1}^n d_l))$, the next definition is motivated:  

\begin{definition}[Gorenstein] \label{gor-def}
We call $M$ \textbf{weakly Gorenstein} if there exists a graded isomorphism of $R$-modules \[ \tau : M \stackrel{\cong}{\longrightarrow} \Hom_k(M,k)(-s)\] for some integer $s$.

If $M$ is weakly Gorenstein and $\tau$ has the property $\tau^*(-s) = \pm \tau$ then we call $M$ \textbf{strongly Gorenstein} or simply \textbf{Gorenstein}. 

\end{definition}

\begin{remark}
More exactly by $\tau^*(-s) = \pm \tau$ means the commutativity of the following diagram:

\small
\xymatrix{ &&&&&&M \ar[r]^{\pm \tau} \ar[d]^{\cong} & M^*(-s) \ar[d]^{=} \\ &&&&&& M^{**} \ar[r]^{\tau^*(-s)} & M^*(-s) }
\normalsize
\end{remark}

\begin{remark}
By Theorem \ref{Gorensteintheorem} a graded $R$-module of finite length is Gorenstein if and only if it can be defined via a symmetric or skew symmetric matrix in divided powers.   
\end{remark}

\begin{remark} \label{remgor}
Let $I \subset R$ a homogeneous zero dimensional ideal. Then $R/I$ is Gorenstein in the usual sense if there is an integer $a$ such that $\Ext^n_R(R/I,R(-d)) \cong R/I(a)$ (for example \cite[Proposition 3.6.11]{BH93}). As $\Ext^n_R(M,R(-d)) \cong \Hom_k(M,k)$ ($M$ is of finite length) this fits with our definition. In the same manner by Macaulay's Theorem \cite{M16} there is a homogeneous element $f$ of degree $-a$ such that $I = \Ann(f)$. Hence $P=(f) \in \Hom_R(R(-a),\kd)$ is the required matrix in divided powers. 
\end{remark}

\section{The Symmetry of the Resolution in the Gorenstein Case}

\label{sec3}
At the beginning of this section it is necessary to fix a notation for the Koszul complex. Again throughout the section let $M$ be a graded $R$-module of finite length.

\begin{notation}[Koszul complex] \label{not1}
Let $R=\Sym(V)$ be the weighted polynomial ring with $V= \langle x_1,\ldots,x_n \rangle_k$. Let $W$ be the dual $k$-vectorspace such that $V= \Hom_k(W,k)$. Choose a basis $(\dx_1,\ldots,\dx_n)$ of $W$ such that the $(\dx_l)$ are a dual to $(x_l)$, and set $\deg \dx_l = d_l$.
\\
The map $(x_l \neg \,\,\,): \bigwedge^i W \rightarrow \bigwedge^{i-1} W$ defined on generators by 
\[
(x_l \neg w_1\wedge \ldots \wedge w_i) = \sum_{j=1}^i (-1)^{j-1} x_l(w_j) w_1\wedge \ldots \wedge \widehat{w_j} \wedge \ldots \wedge w_i
\]
is called {\it contraction}. View $R\otimes_k \bigwedge^i W$ as graded free $R$-module by left multiplication, the grading is given by $(R \otimes_k \bigwedge^i W)_j = \bigoplus_{j_1 + j_2=j} R_{j_1} \otimes_k (\bigwedge^i W)_{j_2}$. Then we define 
$\delta_i : R\otimes_k \bigwedge^i W \rightarrow R\otimes_k \bigwedge^{i-1} W$ on generators by $r\otimes w \mapsto \sum_{l=1}^n r x_l \otimes (x_l \neg w)$.
 \\
We call the  graded complex 
\[
K(x): \,\,\, 0 \rightarrow R\otimes_k \bigwedge^n W \stackrel{\delta_n}{\rightarrow} \ldots \rightarrow R \otimes_k \bigwedge^2 W \stackrel{\delta_2}{\rightarrow} R \otimes_k \bigwedge^1 W \stackrel{\delta_1}{\rightarrow} R \otimes_k \bigwedge^0 W \rightarrow 0
\]
the {\it Koszul complex}. It has only homology at $H^0(K(x)) =k$. For the definition and the homology statement see \cite[Section 1.6]{BH93}.  
\end{notation}

We recall two well known properties of the Koszul complex. The used tools are essential for our main theorem. 

\begin{lemma and definition}[Selfduality of the Koszul Complex I] \label{kosI}
Let $1 \leq i \leq n$. Consider $\alpha_i : R \otimes \bigwedge^i W \rightarrow (R \otimes \bigwedge^{n-i} W)^{\vee}:$
\[
r \otimes w \mapsto \left \{ \begin{array}{c} R \otimes \bigwedge^{n-i} W \rightarrow R \otimes \bigwedge^n W \\ (r'\otimes w') \mapsto (r'\cdot r) \otimes (w' \wedge w), \end{array} \right.
\]
where $(\quad)^{\vee} = \Hom_R(\quad, R \otimes \bigwedge^n W )$. We denote by $(K(x))^{\vee}$ the dual Koszul complex, i.e. the complex resulting from applying $(\quad)^{\vee}$ to $K(x)$. Let $\ell(i)=\lfloor \frac{i-1}{2} \rfloor$. Then the following diagram is graded commutative (as the coherent diagram for $i=1,\ldots,n$):

\medskip

\small
\xymatrix{
& K(x): & \ldots \ar[r] & R\otimes \bigwedge^{i}W \ar[rr]^{\delta_i} \ar[d]^{(-1)^{n-i} (-1)^{\ell(i)}\alpha_i} && \ar[r] R \otimes \bigwedge^{i-1} W \ar[d]^{(-1)^{\ell(i)}\alpha_{i-1}}&  \ldots \\
& K(x)^{\vee}: & \ldots \ar[r] &(R \otimes \bigwedge^{n-i}W)^{\vee} \ar[rr]^{\delta_{n-i+1}^{\vee}} & &\ar[r] (R \otimes \bigwedge^{n-i+1} W)^{\vee} &  \ldots .\\
}
\normalsize
\end{lemma and definition}

\begin{proof}
We prove the claim on generators $R \otimes \bigwedge^i W \ni r \otimes w:$

\[
r \otimes w  \stackrel{\alpha_i}{\mapsto} \left \{ \begin{array}{c} R \otimes \bigwedge^{n-i} W \rightarrow R \otimes \bigwedge^n W \\  r' \otimes w' \mapsto r'r \otimes w' \wedge w \end{array}  \right \} \stackrel{\delta^{\vee}_{n-i+1}}{\longmapsto} \left \{ \begin{array}{c} R \otimes \bigwedge^{n-i+1} W \rightarrow R\otimes \bigwedge^n W \\
r'' \otimes w'' \mapsto \sum_{l=1}^n r'' r x_l \otimes (x_l \neg w'')\wedge w \end{array} \right \}.\]
And the other way around:
\[
(r\otimes w) \stackrel{\delta_i}{\mapsto} \sum_{l=1}^n  r x_{l} \otimes (x_l \neg w) \stackrel{\alpha_{i-1}}{\mapsto} \left \{ \begin{array}{c} R \otimes \bigwedge^{n-i+1} W \rightarrow R \otimes \bigwedge^n W \\
(r'' \otimes w'') \mapsto \sum_{l=1}^n r'' r x_l \otimes w'' \wedge (x_l \neg w). \end{array} \right.\] 

Now the claim follows as $w'' \wedge ( x_{l} \neg w) = (-1)^{n-i} (x_{l} \neg w'') \wedge w$. This explains the $(-1)^{n-i}$, it requires a sign change at every second down arrow. The $(-1)^{\ell(i)}$ give the actual sign at the first down arrow of a square.  
\end{proof}

This leads directly to some symmetry result of the Koszul complex:

\begin{lemma}[Selfduality of the Koszul Complex II]
Let $n$ be odd and $m=\frac{n-1}{2}$. Consider the complex 
\small
\[
K: \,\,\, 0 \rightarrow (R\otimes \bigwedge^0 W)^{\vee} \stackrel{\delta_1^{\vee}}{\rightarrow} \ldots (R\otimes \bigwedge^m W)^{\vee} \stackrel{\delta_{m+1}\circ \alpha^{-1}_{m+1}}{\rightarrow} R\otimes \bigwedge^m W \rightarrow  R \rightarrow \ldots \stackrel{\delta_1}{\rightarrow}  R \otimes \bigwedge^0 W \rightarrow 0
\]
\normalsize
with $(\quad)^{\vee} = \Hom_R(\quad,R\otimes \bigwedge^n W )$.
\\
If $n \equiv 3 \mod 4$ then the complex $K$ is skew symmetric. That is $\delta_{m+1} \circ \alpha^{-1}_{m+1}$ is skew with respect to dual bases. 
\\
In the same manner $K$ is symmetric if $n \equiv 1 \mod 4$. 
\end{lemma}

\begin{proof}
The proof is straight forward.
\end{proof}

We give two strongly related resolution constructions for graded $R$-modules $M$ of finite length. The first one was originally given by Nielsen (\cite{N80}). The construction is needed in the proof of the main theorem. Note that it is in general non minimal. 

\begin{theorem and construction}[Nielsen I] 
\label{complexM} 
Let $A_i(M) = R \otimes_k \bigwedge^i W \otimes_k M$ be viewed as a graded $R$-module by left multiplication (i.e. $r' \cdot r \otimes w \otimes m = (r'r)\otimes w \otimes m$). Define the $j$-th graded part $(A_i(M))_j := \bigoplus_{j_1+j_2+j_3=j} R_{j_1} \otimes_k (\bigwedge^i W)_{j_2} \otimes_k M_{j_3}$.  
\\
We define $\phi_i : A_i(M) \rightarrow A_{i-1}(M)$, $i=1,\ldots,n,$ on generators by 
\[
r \otimes w \otimes \mm \mapsto \sum_{l=1}^n x_l  r \otimes (x_l \neg w)\otimes \mm - \sum_{l=1}^n r \otimes (x_l \neg w) \otimes (x_l \cdot \mm),
\]
for all $r \in R, w \in \bigwedge^i W$ and $\mm \in M$. By $x_l \cdot \mm$ is meant the multiplication in $M$. 
\\
Then
\[
 A_0(M) \stackrel{\phi_1}{\leftarrow} A_1(M) \leftarrow \ldots \leftarrow A_{n-1}(M) \stackrel{\phi_{n}}{\leftarrow} A_n(M) \leftarrow 0
\]
is a graded complex of free $R$-modules.
\end{theorem and construction}

\begin{proof}
Let $\phi_{i,0}$ be defined by the special part of $\phi_i$ from above given by $r \otimes w \otimes \mm \mapsto  \sum_{l=1}^n (x_l  r)\otimes (x_l \neg w) \otimes \mm,$ and $\phi_{i,1}$ by $r \otimes w \otimes \mm  \mapsto - \sum_{l=1}^n r \otimes (x_l \neg w)  \otimes (x_l \cdot \mm).$ The complex from an schematic point of view splits into parts:

\vspace{0.5 cm}

\small
\xymatrix{
  &&& && \ar[dl]_{\phi_{i,1}}  A_{i-1}(M)  &\ar[l]_{\qquad \phi_{i,0}} \ar[dl]_{\phi_{i,1}} A_{i}(M) &  \\
 && &&A_{i-2}(M) &  \ar[l]_{\phi_{i-1,0}} A_{i-1}(M) &&  \\
 }
\normalsize

\vspace{0.5 cm}

Let us show that the above construction really gives a complex. At first we have to see that $\phi_{i,0} \circ \phi_{i+1,0} = 0$, this is because of the corresponding property of the Koszul complex. The same holds for $\phi_{i,1} \circ \phi_{i+1,1}= 0$. Moreover we should see that $\phi_{i,0} \circ \phi_{i+1,1} = - \phi_{i,1} \circ \phi_{i+1,0}$.
\begin{eqnarray*}
\phi_{i,0}\circ \phi_{i+1,1}(r\otimes w\otimes \mm) =  - \sum_{l_2=1}^n \sum_{l_1=1}^n (x_{l_2} r) \otimes (x_{l_2} \neg(x_{l_1} \neg w)) \otimes (x_{l_1}  \mm) = \\
\sum_{l_2=1}^n \sum_{l_1=1}^n (x_{l_1} r) \otimes (x_{l_2} \neg(x_{l_1} \neg w)) \otimes (x_{l_2} \mm) = - \phi_{i,1} \circ \phi_{i+1,0}(r \otimes w \otimes \mm),
\end{eqnarray*}
as $(x_{l_2} \neg( x_{l_1} \neg w)) = - (x_{l_1} \neg(x_{l_2} \neg w))$. 
\end{proof}

\medskip

\begin{remark} \label{remdouble}
If $R$ is the trivially weighted polynomial ring, one can define $A_{(-j-i,j)}(M):= R\otimes \bigwedge^i W \otimes M_j$. Then the built complex is the total complex of the double complex obtained from
\[
\ldots \leftarrow A_{(-j-i+1,j)}(M) \stackrel{\phi_{i,0}}{\leftarrow} A_{(-j-i,j)}(M) \stackrel{\phi_{i+1,0}}{\leftarrow} A_{(-j-i-1,j)}(M) \leftarrow \ldots 
\]
and
\[
\ldots \leftarrow A_{(-j-i,j+1)}(M) \stackrel{\phi_{i,1}}{\leftarrow} A_{(-j-i,j)}(M) \stackrel{\phi_{i+1,1}}{\leftarrow} A_{(-j-i,j-1)}(M) \leftarrow \ldots. 
\]
(Compare \cite[1.5]{N80})
\end{remark}

The second complex construction respects the $R$-module structure of $M$ in a different way. Before stating it we need to show a lemma at first. The setup is the following:

\begin{notation}[Diagonal Action]  \label{diagaction}
Let $F$ be a graded $R$-module. Let $A=F \otimes_k M$ be the $k$-vectorspace product. 
\\
Then we consider the $R$-module $\dd A= \dd(F\otimes_k M)$ with  the {\it diagonal action} \[\Delta(x): \quad x \cdot e\otimes \mm := (1\otimes x)+ (x \otimes 1) \, e\otimes \mm := e \otimes x \mm + x e \otimes \mm\] for all $x \in V\setminus{0}$,  $e \in F$, $\mm \in M$,  and its linear extension. If the context is clear then the symbol $\dd$ should indicate that we mean a module built with respect to this diagonal action. $\dd A$ is graded with respect to $(\dd A)_j = \bigoplus_{j_1+j_2 =j} F_{j_1} \otimes M_{j_2}$.  
\end{notation}

\begin{lemma} \label{freeness}
Let $F$ be a graded free $R$-module, and let $(f_1,\ldots,f_{\mu})$ be a homogeneous basis of $F$.
 Chosen a homogeneous $k$-vectorspace basis $(\mm_1,\ldots,\mm_{\nu})$ of $M$ we get that
  \[B = ( f_{i}\otimes \mm_j)_{(i,j)}\] is a homogeneous basis of $\dd A$,i.e. $\dd A$ is free $R$-module of $\rank \dd A = \mu \cdot \nu$.
\end{lemma}

\begin{proof}
The proof is an induction using the grading of $M$. 
\end{proof}

We have the following free resolution construction of $M$ as graded $R$-module:

\begin{theorem and construction}[Nielsen II] 
\label{reslemma}
Let $\dd A_i(M)= \dd (R \otimes \bigwedge^i W \otimes M)$ be the graded $R$-module with the diagonal action $x \cdot (r \otimes w\otimes \mm) := (x r) \otimes w \otimes \mm + r \otimes w \otimes (x \mm)$ for all $x \in V \setminus \{0\}$ and its linear extension. Then
\[
0 \leftarrow M= \dd(k \otimes M) \stackrel{\leftarrow}{\pr}  \dd(A_0(M)) \stackrel{\leftarrow}{\dd(\phi_1)}  \dd(A_1(M)) \leftarrow \ldots \stackrel{\leftarrow}{\dd(\phi_n)}  \dd(A_n(M)) \leftarrow 0
\]
is a graded free resolution of $M$ where $\dd(\phi_i)$ is defined as in the Koszul complex by  \[r \otimes w \otimes \mm \mapsto  \sum_{l=1}^n (x_l r) \otimes (x_l \neg w) \otimes \mm.\] 
\end{theorem and construction}

\begin{proof}
First of all the differentials are linear: $\dd(\phi_i)(x \cdot (r\otimes w\otimes \mm)) = \dd(\phi_i)(xr \otimes w \otimes \mm + r \otimes w \otimes (x\mm)) = \sum_{l=1}^n x_l x r \otimes (x_l \neg w) \otimes \mm + \sum_{l=1}^n x_l r \otimes (x_l \neg w) \otimes (x \mm)  = x \cdot \dd(\phi_i)(r \otimes w \otimes \mm)$ for all $x \in V \setminus \{ 0\}$. The exactness follows by the exactness of the Koszul complex $K(x)$, as the differential behaves as $\dim_k M$ copies of the Koszul differentials. Moreover we respect to the $R$-module structure of $M$ as $x \cdot \mm = (1 \otimes x)+(x \otimes 1) (1 \otimes \mm) =   1 \otimes  x\mm + x \otimes \mm = 1 \otimes x \mm$ in $\dd(k\otimes_k M) \cong \dd A_0 / \Image \dd(\phi_1)$ for all $x \in V$. 
\end{proof}

Using the dual Koszul complex an analogous construction is obtained. 

\begin{corollary}[Nielsen IIa]
\label{resIIa} Let $(\quad)^{\vee} = \Hom_R(\quad, R\otimes \bigwedge^n W)$, and let $B_i(M)= (R \otimes \bigwedge^i W)^{\vee} \otimes M$ be the graded free module with the left multiplication. Again $\dd(B_i(M)) = \dd ((R \otimes_k \bigwedge^i W)^{\vee} \otimes_k M)$ is the module together with the diagonal action. Let $\delta_i^{\vee}: (R \otimes_k \bigwedge^{i-1} W)^{\vee} \rightarrow (R \otimes_k \bigwedge^i W)^{\vee}$ be the differentials from the dual Koszul complex. We denote by $\dd(\varphi_i)$ the homomorphisms $\dd(B_{i-1}(M)) \rightarrow \dd(B_i(M))$ defined by $\dmp \otimes \mm \mapsto \delta_i^{\vee}(\dmp) \otimes \mm$, for all $\dmp \in (R\otimes \bigwedge^{i-1} W)^{\vee}$ and $\mm \in M$. Fixing the canonical $k$-vectorspace  isomorphism \[f: \Hom_k(\bigwedge^n W, \bigwedge^n W) \stackrel{\cong}{\rightarrow} k, \id \mapsto 1,\]  
\[
0 \leftarrow M = \dd(k \otimes M) \stackrel{f \otimes \id}{\leftarrow} \dd(B_n(M)) \stackrel{\dd(\varphi_{n})}{\leftarrow}  \ldots \leftarrow \dd(B_1(M)) \stackrel{\dd(\varphi_1)}{\leftarrow} \dd(B_0(M)) \leftarrow 0
\]
is a graded free resolution of $M$. 
\end{corollary}

\begin{remark} 
Both complex contructions (Nielsen I and Nielsen II) are isomorphic in a canonical way. Let $i \geq 1$. Then the following diagrams commute for all $i$:

\begin{center}
\xymatrix{
&&&&&& A_i(M)  \ar[r]^{\epsilon_i} \ar[d]^{ \phi_{i}} &  \dd(A_i(M)) \ar[d]^{\dd(\phi_{i})} \\
&&&&&&A_{i-1}(M) \ar[r]^{\epsilon_{i-1}} & \dd(A_{i-1}(M)),\\
}
\end{center}

\medskip

where $\epsilon_i: A_i(M) \rightarrow \dd(A_i(M))$ is defined by $r\otimes w \otimes \mm \mapsto r (1\otimes w \otimes \mm)$ for all $\mm \in M, w \in \bigwedge^iW$ and $r \in R$.
\end{remark}

\begin{corollary}[Nielsen I]
The first complex construction (Nielsen I), 
\[
0 \leftarrow M \leftarrow A_0(M) \stackrel{\phi_1}{\leftarrow} A_1(M) \leftarrow \ldots \leftarrow A_{n-1}(M) \stackrel{\phi_{n}}{\leftarrow} A_n(M) \leftarrow 0,
\]
is a graded free resolution of the given graded finite length module $M$. 
\end{corollary}

We prove now the selfduality of the resolution. 

\begin{theorem}[Selfduality of the Resolution] 
\label{maintheorem} 
Let $n$ be odd, $m = \frac{n-1}{2}$. Let $M$ Gorenstein with $\tau = \pm \tau^*(-s)$, $(\quad)^{\vee} = \Hom_R(\quad,R\otimes \bigwedge^n W)(-s)$, and let $(\quad)^*=\Hom_k(\quad,k)$. Recall that $A_i(M) = R \otimes_k \bigwedge^i W \otimes_k M$, and define $\phi_i$ as in Theorem \ref{complexM} (Nielsen I).  Let $\beta_i :  A_i(M) \rightarrow (A_{n-i}(M))^{\vee}$ be defined by 
\[
r\otimes w \otimes \mm \mapsto \left \{ \begin{array}{c} A_{n-i}(M) \rightarrow R \otimes \bigwedge^n W \\
(r' \otimes w' \otimes \mm') \mapsto (r' r) \otimes (w'\wedge w) \otimes \tau(\mm)(\mm'), \end{array} \right. 
\]
for all $r,r' \in R, w \in \bigwedge^i W, w' \in \bigwedge^{n-i} W$ and $\mm, \mm' \in M$. Then the graded free resolution of $M$
\[
K(M): \,\,\, 0 \rightarrow (A_0(M))^{\vee} \stackrel{\phi_1^{\vee}}{\rightarrow} \ldots (A_m(M))^{\vee} \stackrel{\phi_{m+1}\circ \beta^{-1}_{m+1}}{\rightarrow} A_m(M) \rightarrow   \ldots \stackrel{\phi_1}{\rightarrow}  A_0(M) \rightarrow M \rightarrow 0
\]
is symmetric in the following sense: Let $B$ be any homogeneous basis of $A_m(M)$ and $B^{\vee}$ its dual basis of $(A_m(M))^{\vee}$. Then the matrix representation of $\phi_{m+1} \circ \beta^{-1}_{m+1}$ has the property:

\[
  M^{B^{\vee}}_B(\phi_{m+1} \circ \beta^{-1}_{m+1}) = (-1)^m \pm  (M^{B^{\vee}}_B(\phi_{m+1} \circ \beta^{-1}_{m+1}))^t .
\]
\end{theorem}

\begin{proof}
Consider $\beta_{m+1} : A_{m+1}(M) \stackrel{\beta_{m+1}}{\rightarrow} (A_{m}(M))^{\vee}$ and $\beta_m: A_{m}(M) \stackrel{\beta_m}{\rightarrow} (A_{m+1}(M))^{\vee}$ and check that $\beta_{m+1}^{\vee} = \pm \beta_m$: First of all
\[
\beta_{m+1}^{\vee} :  A_m(M) \stackrel{\cong}{\rightarrow} (A_m(M))^{\vee \vee} \stackrel{\beta_{m+1}^{\vee}}{\rightarrow} (A_{m+1}(M))^{\vee},
\]
\[
r'\otimes w' \otimes \mm' \mapsto ( \gamma \mapsto \gamma(r' \otimes w' \otimes \mm')) \mapsto \left \{ \begin{array}{c} A_{m+1}(M) \rightarrow R\otimes \bigwedge^n W, \\  r \otimes w \otimes \mm \mapsto (r' \cdot  r) \otimes (w' \wedge w) \otimes \tau(\mm)(\mm'). \end{array} \right.
\]
On the other hand $\beta_m$ maps as follows:
\[
r'\otimes w' \otimes \mm' \mapsto  \left \{ \begin{array}{c} A_{m+1}(M) \rightarrow R\otimes \bigwedge^n W, \\ (r \otimes w \otimes \mm) \mapsto (r' \cdot r) \otimes (w \wedge w') \otimes \tau(\mm')(\mm). \end{array} \right.
\]
As $w \wedge w' = (-1)^{m (m+1)} w' \wedge w = w' \wedge w$ and $\tau(\mm')(\mm) = \pm \tau(\mm)(\mm')$ by the Gorenstein property the claim follows. 
\\
Secondly we verify that the following diagram is commutative:

\xymatrix{
&&& R\otimes \bigwedge^{m+1}W \otimes M \ar[r]^{\beta_{m+1}} \ar[d]^{ \phi_{m+1}} &  (R \otimes \bigwedge^{m} W \otimes M)^{\vee} \ar[d]^{(-1)^m \phi_{m+1}^{\vee}} &(*)\\
&&&R \otimes \bigwedge^{m}W \otimes M \ar[r]^{ \beta_m} & (R \otimes \bigwedge^{m+1} W \otimes M )^{\vee}.&\\
}
\[
r\otimes w \otimes \mm \stackrel{\beta_{m+1}}{\mapsto}  \left \{ \begin{array}{c} A_{m+1}(M) \rightarrow R\otimes \bigwedge^n W, \\  (r'\otimes w' \otimes \mm') \mapsto (r' r)\otimes (w' \wedge w) \otimes \tau(\mm)(\mm'), \end{array} \right. 
\]
which is mapped further by $\phi^{\vee}_{m+1}$ to an element of $(A_{m+1}(M))^{\vee}$. This is the map $A_{m+1}(M) \rightarrow R \otimes \bigwedge^n W:$
\small
\[
(r''\otimes w'' \otimes \mm'') \mapsto \sum_{l=1}^n (x_l r'' r)\otimes ((x_l \neg w'')\wedge w) \otimes \tau(\mm)(\mm'') -   
\sum_{l=1}^n (r'' r)\otimes ((x_l \neg w'')\wedge w) \otimes \tau(\mm)(x_l \mm'').  
\]
\normalsize
The other way around we gain
\begin{eqnarray*}
r\otimes w \otimes \mm \stackrel{\phi_{m+1}}{\mapsto} \sum_{l=1}^{n} (x_l r) \otimes(x_l \neg w) \otimes \mm - \sum_{l=1}^n r \otimes (x_l \neg w) \otimes (x_l \mm) \stackrel{\beta_m}{\mapsto} \\ \left \{ \begin{array}{c} A_{m+1}(M) \rightarrow R\otimes \bigwedge^n W, \\ (r''\otimes w'' \otimes \mm'')  \mapsto  \sum_{l=1}^{n} (x_l r r'') \otimes (w'' \wedge(x_l \neg w)) \otimes \tau(\mm)(\mm'') - \end{array} \right. \\ \left \{ \begin{array}{c} \\\sum_{l=1}^n r r''\otimes (w''\wedge (x_l \neg w)) \otimes \tau(x_l \mm)(\mm''). \end{array} \right.
\end{eqnarray*}
Because of $\tau(\mm)(x_l \mm'') = \tau(x_l \mm)(\mm'')$ the commutativity comes down to see that 
$(w'' \wedge (x_l \neg w)) = (-1)^m ((x_l \neg w'') \wedge w).$ But this is well known from the Koszul complex. 

\medskip

Now it follows with the first claim and the commutativity of (*) that
\begin{eqnarray*}
M^{B^{\vee}}_B(\phi_{m+1} \circ \beta_{m+1}^{-1})^t = M^{B^{\vee}}_B((\phi_{m+1} \circ \beta_{m+1}^{-1})^{\vee}) = M^{B^{\vee}}_B((\beta_{m+1}^{-1})^{\vee} \circ \phi^{\vee}_{m+1}) = \\ \pm M^{B^{\vee}}_B((\beta_m^{-1}) \circ \phi^{\vee}_{m+1})) =  \pm (-1)^m M^{B^{\vee}}_B(\phi_{m+1} \circ \beta_{m+1}^{-1}).
\end{eqnarray*}
\end{proof}

We can easily obtain a minimal free resolution  with the same symmetry properties as in Theorem \ref{maintheorem}:

\begin{corollary} \label{minres} \index{Minimal selfdual resolution} \index{Gorenstein module! has minimal symmetric resolution} 
Let $n \geq 3$ be an odd integer, $m = \frac{n-1}{2}$, and let $k$ be $\car k \neq 2$. Let $d=\sum_{l=1}^n d_l$, let $f: \bigwedge^n W \rightarrow k(-d)$ be a fixed isomorphism. Let $s$ be the top degree of M and $(\quad)^{\vee}= \Hom_R(\quad,R(-d -s))$. If $M$ is Gorenstein then there is a minimal graded free resolution of $M$,  
\[
0 \leftarrow M \leftarrow F_0 \stackrel{\psi_1}{\leftarrow} F_1 \leftarrow \ldots \leftarrow F_{m}\stackrel{\psi_{m+1}}{\leftarrow}  (F_{m})^{\vee} \leftarrow \ldots \leftarrow  \stackrel{\psi_1^{\vee}}{\leftarrow} (F_0)^{\vee} \leftarrow 0,
\] 
with the same symmetry properties as in Theorem \ref{maintheorem}. 
\end{corollary}

\begin{proof}
The proof uses a symmetric minimization process such that the symmetry of the resolution is kept. To do so we need $\car k \neq 2$. 
\end{proof}

If we talk about a Gorenstein ideal in the following we refer to the usual definition. 

\begin{corollary}[Zero Dimensional Gorenstein Ideals] \index{Gorenstein ideal! symmetric resolution}  
Let $n$ be odd, $m=\frac{n-1}{2}$, $d=\sum_{l=1}^n d_l$and $\car k \neq 2$. 
Let $I$ be a homogeneous zero dimensional Gorenstein ideal in $R$ with top degree $s$. Then there is a minimal graded free resolution of $R/I$, 
\[
0 \leftarrow R/I\leftarrow R \stackrel{\psi_1}{\leftarrow} F_1 \leftarrow \ldots F_m \stackrel{\psi_{m+1}}{\leftarrow}  F_{m}^{\vee} \leftarrow \ldots F_1^{\vee} \stackrel{\psi_{1}^{\vee}}{\leftarrow}  R(-d-s) \leftarrow 0,
\] 
 such that $\psi_m$ is skew symmetric if $m$ is odd and symmetric if $m$ is even with respect to dual bases. 
\end{corollary}

\begin{proof}
We apply our main Theorems \ref{maintheorem} respectively \ref{minres}. By Remark \ref{remgor} the isomorphism $\tau: R/I \rightarrow \Hom_k(R/I,k)(-s)$ such that $\tau^*(-s) = \tau$ exists. 
\end{proof}

\begin{remark}
There are intersectional cases with the Theorem of Buchsbaum and Eisenbud (\cite[Theorem 2.1]{BE77}). We do not use the structure as a differential graded algebra on the resolution of $R/I$, that is why we are not restricted on the codimension. 
\end{remark}

\section[Some Applications]{Some Applications}
\label{sec5}
The structure Theorem \ref{maintheorem} has a bunch of natural applications. Let $k$ be again any field if not stated differently. Moreover let $R=k[x_1,\ldots, x_n]$ be the usual (trivially weighted) polynomial ring. Let $(\quad)^* = \Hom_k(\quad,k)$.  Moreover we use the notation from the previous section fixed in \ref{not1} --- in the case $d_1 = \ldots = d_n =1$.

\begin{remark}
Let $c$ be a positive integer with $c \geq 2$. Let $A \in k^{c \times c}$ such that $A^t=-A$ and $A_{(j,j)}=0$ for all $j \in \{1,\ldots, c\}$, i.e. $A$ is skew symmetric. Then $A$ is of even rank. 
\end{remark}

\begin{corollary} \label{3cor}\label{char2cor}\label{cor1}
Let $n = 2^{\ell} - 3$, $\ell \geq 3$ and let $m = \frac{n-1}{2}=  2^{\ell -1} -2$. Let $k$ be a field of characteristic $\car k = 2$. Let $I \subset R$ be a graded Artinian Gorenstein ideal with Hilbert function $(1,n,n,1)$.
\\
Then the graded minimal free resolution of $R/I$ has a Betti table of type
 \[
\begin{array}{cccccccccc}
&\vline&0&1&&m&m+1&&&n \\ \hline
&\vline&1&- &&   &  &&&\\
&\vline&-&\beta_{1,2}& \cdots&\beta_{m,m+1} & 2a+1 &\cdots & \beta_{1,3}&-\\ 
&\vline&- &\beta_{1,3}& \cdots & 2a+1 & \beta_{m,m+1}&\cdots &\beta_{1,2}&-\\
&\vline&-& &  &&&&-&1\\
\end{array}
\] 
for some $a \in \{0,\ldots, \lfloor \frac{1}{2} n \binom{n}{m} \rfloor - \binom{n}{m-1}-1\}$. 
\end{corollary}

\begin{proof}
It is clear that $R/I$ is also Gorenstein in the sense of Definition \ref{gor-def}. We apply the Nielsen Construction of Theorem \ref{maintheorem} with respect to a certain basis $B$. Let $W_m=(1\otimes \chi_{i_1}\wedge \ldots \wedge \chi_{i_m})_{(i_1,\ldots,i_m)}$ be the canonical basis of $R\otimes \bigwedge^m W$, and let $\widetilde{B}$ be any homogeneous $k$-vectorspace basis of $R/I$. Let $B=W_m \otimes \widetilde{B}$. Let $f$ be the restriction of the representation matrix of the middle map $\phi_{m+1}\circ \beta_{m+1}^{-1}$ to the elements $B_{2}$ of $B$ in degree $2+m$ and their duals $B_{2}^{\vee}$ in $B^{\vee}$.  By Theorem \ref{maintheorem} we know that $f$ has the property $f = f^t$.
\\
Let $\mb= 1\otimes \chi_{i_1} \wedge \ldots \wedge \chi_{i_m} \otimes \mm$ be any element of $B_{2}$, and let $\db$ be its dual element in $B_{2}^{\vee}$. Then $\db$ is up to sign of the form $\db=(\_ \wedge \chi_{i_{m+1}} \wedge \ldots \wedge \chi_{i_n}) \otimes \dm :=$ \[ \left \{ \begin{array}{c} R\otimes \bigwedge^m W \otimes M_{p+1} \rightarrow R \otimes \bigwedge^n W \\ r\otimes w \otimes \mm \mapsto r \otimes(w \wedge \chi_{i_{m+1}} \wedge \ldots \wedge \chi_{i_n}) \otimes \dm(\mm), \end{array} \right.\]  such that $i_1, \ldots, i_n \in \{1,\ldots,n\}$ are pairwise different. We obtain $\phi_{m+1} \circ \beta_{m+1}^{-1}(\db)= \phi_{m+1}(\chi_{i_{m+1}}\wedge \ldots \wedge \chi_{i_n} \otimes \tau^{-1}(\dm)) = x_{i_{m+1}} \otimes \chi_{i_{m+2}} \wedge \ldots \wedge \chi_{i_{n}} \otimes \tau^{-1}(\dm) + \ldots + (-1)^{n-m-1} x_{i_n} \otimes \chi_{i_{m+1}} \wedge\ldots \wedge \chi_{i_{n-1}} \otimes \tau^{-1}(\dm) - 1 \otimes \chi_{i_{m+2}} \wedge \ldots \wedge \chi_{i_n} \otimes (x_{i_{m+1}} \tau^{-1}(\dm))+ (-1)^{n-m-1} \otimes \chi_{i_{m+1}}\wedge \ldots \wedge \chi_{i_{n-1}} \otimes (x_{i_n} \tau^{-1}(\dm)). $ Hence the linear representation of $\phi_{m+1} \circ \beta_{m+1}^{-1}(\db)$ does not involve $\chi_{i_1} \wedge \ldots \wedge \chi_{i_{m}}$, i.e. not $\mb$. This implies the diagonal elements ${f}_{(j,j)}=0$ for all $j$. Therefore $f$ is also skew symmetric as $\car k = 2$, hence of even rank. 

The resolution \ref{maintheorem} of $R/I$ has constant middle maps  of type
\small
\[
\xymatrix{1&\ldots &\binom{n}{m-1}&\binom{n}{m}&\binom{n}{m}&\ar[dl]^{g^t}\binom{n}{m-1}&\ldots &1\\
                 n&\ldots&n\binom{n}{m-1}&n\binom{n}{m}&\ar[dl]^{f} n\binom{n}{m} &n\binom{n}{m-1}& \ldots &n\\
                n&\ldots &n\binom{n}{m-1}&\ar[dl]^{g} n \binom{n}{m}&n \binom{n}{m}&n\binom{n}{m-1}&\ldots &n\\
                1&\ldots&\binom{n}{m-1}&\binom{n}{m}&\binom{n}{m}&\binom{n}{m-1}&\ldots &n.}
\]
\normalsize

 For $0 \leq i \leq 2^{\ell} -3$ we know that $\binom{2^{\ell}-3}{i}$ change parity in every second step, starting with odd. That means $\binom{2^{\ell}-3}{2^{\ell-1}-3}$ is odd and  $\binom{2^{\ell}-3}{2^{\ell-1}-2}$ is even as $2^{\ell-1}-3 \equiv 1 \mbox{ mod } 4$. Hence $\binom{n}{m}$ and $\binom{n}{m-1}$ are of opposite parity. By the assumption on the Hilbert function $g$ and $g^t$ are of rank $\binom{n}{m-1}$. Hence the size $n \binom{n}{m}$ of the matrix $f$ and the rank $\binom{n}{m-1}$ of $g$ are of the opposite parity.  Therefore we can reduce $f$ by operations from the left and right to $\tilde f$, a quadratic matrix of odd size. $\tilde f$ is of even rank. Hence a minimization gives the free submodules of even rank $2a+1$.
\end{proof}

There are two corollaries of special interest. The first one concerns Green's Conjecture in characteristic $2$ for curves of genus $g=2^{\ell}-1$. This case was already determined for smooth curves by Schreyer in \cite{S86} and \cite{S91} by geometric methods while we use exclusively algebraic results. 

\begin{corollary}[Green's Conjecture in Characteristic $2$] \index{Green's Conjecture}
The obvious extension of the Green's Conjecture to positive characteristic fails for general curves of genus $g=2^{\ell}-1$ for all $\ell \geq 3$ in characteristic $2$. 
\end{corollary}

\begin{proof}
In this case Green's conjecture would mean that the minimal free resolution of the canonical model $X \subset \P^{g-1}$ of the curve has a selfdual pure Betti table of type: 
\[
\begin{array}{cccccccccc}
&\vline&0&1&\ldots&(g-3)/2&(g-3)/2 +1&\ldots&&g-2 \\ \hline
&\vline&1&- &&   &  &&&\\
&\vline&-&\beta_{1,2}& \cdots&\beta_{(g-3)/2,(g-3)/2} & - & && \\ 
&\vline&- &&  & -& \beta_{(g-3)/2,(g-3)/2}&\cdots &\beta_{1,2}&\\
&\vline&-& &  &&&&-&1.\\
\end{array}
\] 
Modulo two regular elements this would lead to an Artinian Gorenstein factor ring with Hilbert function $(1,g-2,g-2,1)$ over the polynomial ring $k[x_1,\ldots, x_{g-2}]$ with a pure resolution. This would be a contradiction to Corollary \ref{char2cor}. 
\end{proof}

\begin{remark} \index{Degree sequence} \index{Pure resolution} \index{Monoid of resolutions}
In the following by a {\it degree sequence} is meant a sequence of integers \[ \md=(\md_0 < \md_1 < \ldots < \md_c). \] A graded minimal free resolution of an $R$-module $M$ is called {\it pure with degree sequence $\md$} if all graded Betti numbers $\beta_{i,j}(M) =0$ except when $j = \md_i$. 
\\
Note that for a given degree sequence $\md$ the Betti table of a pure minimal graded free resolution with $\md$ is uniquely determined by the Herzog-K\"uhl equations up to a rational multiple (\cite[Theorem 1]{HK84}). 
\\
It is clear that the Betti tables of graded minimal free resolutions of $R$-modules form a monoid with respect to addition (take the direct sum of the corresponding modules). 
\end{remark}

The second Corollary concerns the Boji-S\"oderberg Conjectures (\cite{BS06}) on the existence of Cohen-Macaulay modules over $R=k[x_1,\ldots,x_n]$ with pure resolutions having any given degree sequence. 
\\
The recent paper of Eisenbud and Schreyer \cite{ES08} gives an introduction and a proof of a strengthened form of the Boji-S\"oderberg Conjectures. Eisenbud and Schreyer present an algorithm which expresses every Betti table of a minimal graded free resolution of a finitely generated graded Cohen-Macaulay module as a positive rational linear combination of the Betti tables of Cohen-Macaulay modules with pure resolutions. That means the Betti tables of Cohen-Macaulay modules over $R$ lie inside a rational cone with Betti tables of pure resolutions on the extremal rays. 
\\
However it is not clear which  Betti tables satisfying the Herzog-K\"uhl equations of a given degree sequence are in the monoid of actual minimal resolutions. Eisenbud and Schreyer conjecture that the monoid of resolutions depends on the characteristic of the base field $k$. We prove this considering the case of $\car k = 2$. 

\begin{corollary}[Monoid of Resolutions of Cohen-Macaulay Modules]
The monoid of resolutions of Cohen-Macaulay graded $R$-modules depends on the characteristic of $k$. 
\end{corollary}

\begin{proof}
Let $R=k[x_1,\ldots,x_5]$. Consider again the case of Artinian Gorenstein factor rings with Hilbert function $(1,5,5,1)$. If $\car(k)=0$ it is easy to construct such a module with betti table 
\[
\begin{array}{cccccccc}
&&1&0 &0  & 0 &0&0\\
&&0&10& 16&0 &0&0\\ 
&&0 &0&  0& 16&10&0\\
&&0&0 & 0 &0&0&1.\\
\end{array}
\] 
By Corollary \ref{cor1} we know there is no Artinian Gorenstein ideal $I \subset R$ over $\car(k)=2$ with such a resolution. Moreover any Cohen-Macaulay module over a polynomial ring in more variables with this Betti table comes modulo a regular sequence down to this situation.
\end{proof}

In his recent paper Erman gives examples for rays of the rational cone of Betti tables with $n-2$ consecutive lattice points --- if $n$ is prime ---, which do not come from Betti tables of actual resolutions \cite[Theorem 1.6(3)]{Erm08}. Similar we can give degree sequences such that there is no Cohen-Macaulay factor ring over any polynomial ring with a pure resolution with these degree sequences, independent of Erman's methods.

\begin{corollary} \label{degreecor}
Let $\ell \geq 2$. Consider the degree sequence
\[
(0,2^{\ell-1}+1, 2^{\ell-1}+2,\ldots,2^{\ell}-1, 2^{\ell}+1,\ldots, 2^{\ell}+2^{\ell-1}-2, 2^{\ell}+2^{\ell-1}-1,2^{\ell+1})
\]
of length $2^{\ell}-1$. 
\\
Then there is no graded Cohen-Macaulay factor ring over the polynomial ring $R$ of codimension $2^{\ell}-1$ with a pure minimal free resolution having this degree sequence. 
\end{corollary}

\begin{proof}
The proof is similar as in \ref{3cor}: We assume the existence of such a factor ring. Modulo a regular sequence we are in the Artinian Gorenstein case. We build our non minimal resolution construction from \ref{maintheorem}. Finally the given numerics leads to a contradiction within the minimization process.  
\end{proof}

\begin{remark}
The first case of Corollary \ref{degreecor} is $(0,3,5,8)$. Here the statement follows also from the Theorem of Buchsbaum and Eisenbud. The next case is $(0,5,6,7,9,10,11,16)$. Computer experiments seem to show that there are Cohen-Macaulay factor rings in any characteristic with this degree sequence and nearly pure resolutions such that $\beta_{i,j}=0$ except when $j=d_j$ and $\beta_{3,8}=\beta_{4,8}=1$.
\end{remark}

\section{Selfdual Resolution implies Gorensteiness}

In this section we want to verify that whenever a graded module $M$ of finite length over the polynomial ring $R$ has a selfdual resolution, there is an $R$-module isomorphism $\tau: M \rightarrow \Hom_k(M,k)(-s)$ such that $\tau^*:=\Hom_k(\tau,k) = \pm \tau(s)$ for some $s \in \mathbb{Z}$. We state the main theorem first and give the proof at the end of the section. 
\\
Again throughout this section let $n$ be a positive integer, $W=\langle \dx_1, \ldots, \dx_n \rangle_k$, $V= \Hom_k(W,k)$, with $(x_l)$ a dual basis to $(\dx_l)$. Let $\deg \dx_l = \deg x_l = d_l > 0$, and $d = \sum_{l=1}^n d_l$. Moreover let $R= \Sym(V)$, the weighted polynomial ring.  
\\
Frequently in this section we need notations for arbitrary elements in graded $R$-modules of finite length $M, N$ and their vectorspace duals $M^*=\Hom_k(M,k), N^*=\Hom_k(N,k)$. We use $\mm \in M, \mn \in N, \dm \in M^*$ and $\dn \in N^*$ to denote these elements if not stated differently.

\begin{theorem} \label{backwardstheorem} 

Let $n$ be odd and let $m =  \frac{n-1}{2}$. Let $M$ be a module of finite length over $R$. Let $(\quad)^{\vee} = \Hom_R(\quad,R(-d))$. We assume that $M$ has a symmetric minimal resolution of the form

\[
  0 \leftarrow M \leftarrow F_0 \stackrel{\psi_1}{\leftarrow} F_1 \leftarrow \ldots \leftarrow F_{m}\stackrel{\psi_{m+1}}{\leftarrow}  F^{\vee}_{m}(-s) \leftarrow \ldots \leftarrow F_1^{\vee}(-s) \stackrel{\psi_1^{\vee}(-s)}{\leftarrow} F_0^{\vee}(-s) \leftarrow 0
\] 
such that $\psi_{m+1}^{\vee} = \pm \psi_{m+1}$ up to twist. Then there exists a graded $R$-module isomorphism 

\[
 \tau : M \rightarrow M^{*}(-s):=\Hom_k(M,k)(-s)
\]

with $\tau^*(-s) = \pm (-1)^m \tau$.

\end{theorem}

The proof of the theorem follows at the end of this section. The following lemma is an essential tool for our machinery within this section.

\begin{lemma} \label{caniso}
Let $F_1$ and  $F_2$ be graded free $R$-modules, and let $M$ be any graded $R$-module of finite length. Let $M^* = \Hom_k(M,k)$. Let $A = F_1 \otimes_k M$ and $B = \Hom_R(F_1,F_2)\otimes_k M^*$.  Recall that $\dd A = \dd(F_1 \otimes_k M)$  and $\dd B= \dd ( \Hom_R(F_1,F_2) \otimes_k M^*)$ are modules with respect to the diagonal action as defined in \ref{diagaction}.   Then there is a canonical  $R$-module isomophism 
\[
\alpha_B:  \dd B \cong \Hom_R(\dd A,F_2).
\]
\end{lemma}

\begin{proof}
As the tensor product, $\Hom_R(\quad,F_2)$ and $\Hom_R(F_1,\quad)$ commute with direct sums, it is enough to see the claim for $F_1=F_2=R$. We define $g_1: B \rightarrow \dd B$ such that $(r \id_R)\otimes \dm \mapsto r (\id_R \otimes \dm),$ for all $r \in R, \dm \in M^*$. And define $g_2: A \rightarrow \dd A$ by $r\otimes \mm \mapsto r (1\otimes \mm),$ for all $\mm \in M$.
\\
By arguments on bases both  maps are obviously isomorphisms. Moreover consider the linear map $\gamma: B \rightarrow \Hom_R(A,R)$ defined by
\[
(r \id_R) \otimes \dm \mapsto \left \{ \begin{array}{c} \dd (R \otimes M) \rightarrow R \\ ( r' \otimes \mm) \mapsto r \id_R( r') \cdot \dm(\mm) = rr' \cdot \dm(\mm). \end{array} \right.
\]
$\gamma$ is again an isomorphism. Hence we can define $\alpha_B: \dd B \rightarrow \Hom(\dd A,R)$ via
 \\
\xymatrix{&&&&& \dd B \ar[r]^{\,\,\quad \alpha_B \qquad \qquad } & \Hom_R(\dd A, R) \ar[d]^{\Hom_R(g_2,R)}\\
          &&&&& B  \ar[u]^{g_1}  \ar[r]^{\gamma \qquad \,\,\,\,\,\,\,\,\,\,}& \Hom_R(A,R) } 
\\
by $\mb \mapsto \Hom_R(g_2^{-1},R)\circ \gamma \circ g^{-1}_1(\mb). $ 
\end{proof}

We denote by $\grMFL$ the category of graded $R$-modules of finite length. Denote by $R\otimes \bigwedge^n W$ the free $R$-module sitting in degree $d$. 
For the proof of Theorem \ref{backwardstheorem} we work out some natural equivalences between the functors $\Ext^n(\quad):=\Ext^n_R(\quad, R \otimes \bigwedge^n W)$ and $(\quad)^*:=\Hom_k(\quad,k)$ on $\grMFL$. In the following let $(\quad)^{\vee}$ be the functor $\Hom_R(\quad, R \otimes \bigwedge^n W)$, and let $P_i= R\otimes_k \bigwedge^i W$ be the free $R$-module by left multiplication. Fix the canonical isomorphism $f: \Hom_k(\bigwedge^n W, \bigwedge^n W) \cong k, \id \mapsto 1$. 
\\
We define a map
\[ r_M : M^*  \rightarrow  \Ext^n(M)\] 
as follows:
\\
We apply the complex construction Nielsen II to $M$ (using the Koszul complex (\ref{reslemma})) and  Nielsen IIa to $M^*$ (using the dual of the Koszul complex (\ref{resIIa})):

\[
0 \leftarrow M^* \leftarrow \dd B_n(M^*) \stackrel{\dd(\varphi_{n})}{\leftarrow} \dd B_{n-1}(M^*) \leftarrow \ldots \leftarrow \dd B_0(M^*) \leftarrow 0. 
\]
We define $r_M$ via the following diagram using the canonical isomorphisms from Lemma \ref{caniso}. We set $\alpha_{M,i}:= \alpha_{P_i^{\vee} \otimes M^*}$. Recall that $B_i(M^*) = P_i^{\vee} \otimes M^*$ and $A_i(M)=P_i \otimes M$: 

\small
\[
\begin{array}{ccccccccc}
0  \leftarrow M^* &\leftarrow \dd B_n(M^*) & \stackrel{\dd(\varphi_{n})}{\leftarrow}  \dd B_{n-1}(M^*)  & \ldots &\stackrel{\dd(\varphi_{1})}{\leftarrow}  \dd B_0(M^*) &\leftarrow 0 \\
 \qquad \quad \downarrow^{r_M} &\qquad \downarrow^{\alpha_{M,n}}&  \qquad  \downarrow^{\alpha_{M,n-1}} &  &  \qquad  \downarrow^{\alpha_{M,0}}&&\\
0  \leftarrow \Ext^n(M)  &\leftarrow (\dd(A_n(M))^{\vee}&  \stackrel{{(\dd\phi_n)}^{\vee}}{\leftarrow} (\dd(A_{n-1}(M))^{\vee} & \ldots   & \stackrel{{(\dd \phi_1)}^{\vee}}{\leftarrow}   (\dd(A_0(M))^{\vee} &\leftarrow 0. \\
\end{array}
\]
\normalsize It makes sense to use the Koszul complex and its dual in the definition: Only in this way we are able to apply the canonical isomorphisms from Lemma \ref{caniso}.

\begin{lemma}
 $r_M$ is well defined as the diagram is commutative. 
\end{lemma}

\begin{proof}
The proof is an immediate calculation. 
\end{proof}

Now we can explicitly define the following natural equivalence of the functors $(\quad)^*$ and $\Ext^n(\quad)$. 

\begin{theorem and definition}[Equivalences of Functors I] \label{r-iso}
Consider the category $\grMFL$. The collection of isomorphism $\{ M \mapsto r_M  \,|\, M \in \Obj(\grMFL)\}$ as defined above gives an isomorphism of the functors $(\quad)^*$ and $\Ext^n(\quad)$, i.e. for all $M, N \in \Obj(\grMFL)$ and all $\tau \in \Mor(\grMFL)$, $\tau : M \rightarrow N$, the diagram 
\[
\xymatrix{& N^{*}  \ar[r]^{\,\,\quad \tau^*  } \ar[d]^{r_N} & M^*\ar[d]^{r_M}\\
          & \Ext^n(N) \ar[r]^{\, \Ext^n(\tau) }& \Ext^n(M) } 
\]
commutes.

\end{theorem and definition}

\begin{proof}
The proof is immediate by diagram chasing. 
\end{proof}

We need to define two other equivalences of functors.

\begin{lemma and definition}[Equivalences of Functors II]
\label{s-iso} We consider on $\grMFL$ two collections of maps. Let $M \in \Obj(\grMFL)$, then we define
\[
s_M : M \rightarrow \Ext^n(\Ext^n(M))
\]
to be the canonical isomorphism. It is computed by any graded free resolution of $M$ and its double dual. 
\\
Moreover we define $u_M: M \rightarrow M^{**}$ by $\mu \mapsto (\phi \mapsto \phi(\mu))$. 
\\
Both collections  $\{ M \mapsto s_M \mbox{ s.th. } M \in \Obj(\grMFL)\}$ and $\{ M \mapsto u_M \mbox{ s.th. } M \in \Obj(\grMFL)\}$ give obviously isomorphisms of the functors $\id$ and $\Ext^n(\Ext^n(\quad))$, respectively $\id$ and $((\quad)^*)^*$.  
\end{lemma and definition}

\medskip

\begin{corollary}[Equivalences of Functors III]
\label{t-iso}  In $\grMFL$ the collection $\{M \mapsto t_M \mbox{ s.th. } M \in \Obj(\grMFL)\}$, with
\[
t_M := \Ext^n(r_M) \circ s_M: M \rightarrow \Ext^n(M^*), 
\]
is an isomorphism of functors: $\id \rightarrow \Ext^n((\quad)^*)$. 
\end{corollary}

\begin{proof}
Let $M, N \in \Obj(\grMFL)$, and let $\tau \in \Mor(\grMFL)$, $\tau: M \rightarrow N$, then the following diagram commutes:
\[
\xymatrix{& \Ext^n(N^{*})  &&\ar[ll]^{\,\,\quad\Ext^n(\tau^*)  }  \Ext^n(M^*)\\
          & \Ext^n(\Ext^n(N)) \ar[u]^{\Ext^n(r_N)} &&\ar[ll]^{\,\,\, \Ext^n(\Ext^n(\tau)) } \Ext^n(\Ext^n(M)) \ar[u]^{\Ext^n(r_M)} \\
          & N \ar[u]^{s_N}&& \ar[ll]^{\tau} M .\ar[u]^{s_M}} 
\]

This is true as the upper part is just $\Ext^n(\quad)$ of the diagram from  \ref{r-iso}.
\end{proof}

\bigskip

The following theorem describes the connection between the two functors $(\quad)^*=\Hom_k(\quad,k)$ and $\Ext^n(\quad) = \Ext^n_R(\quad,R\otimes_k \bigwedge^n W)$ in $\grMFL$. It is central for the proof of Theorem \ref{backwardstheorem}. We state it in the language of the categories from the above isomorphisms of functors (\ref{r-iso}, \ref{s-iso}, \ref{t-iso}).  

\begin{theorem} \label{functortheorem}
Let $n$ be odd, and let $m =  \frac{n-1}{2}$, and let $M,N \in \Obj(\grMFL)$ and $\tau \in \Mor(\grMFL)$, $\tau : M \rightarrow N$.  Using the isomorphisms of functors from above the following diagram commutes:
\vspace*{0.5cm}
\[
\xymatrix{& M^*      \ar[r]^{r_M \qquad } & \Ext^n(M) \\
          & N^*  \ar[u]^{(-1)^m \tau^* \,\,}  \ar[r]^{t_{N^*} \qquad }& \Ext^n(N^{**}) \ar[u]_{\Ext^n(\tau)\circ \Ext^n(u_N)}.  } 
\]
\end{theorem}

\begin{remark}
Besides the technical details the central point of this theorem is the following: For the definition of $r_M$ we resolve $M$ via the Koszul complex, and $M^*$ via the dual Koszul complex. Moreover $t_{N^*}$ --- at least in the case $N=M^*$ --- is roughly $\Ext^n(\quad)$ of $r_M$. That is why we have to resolve this time $N^*$ via the Koszul complex and $N^{**}$ via its dual. The nature of the Koszul complex finally gives the sign. 
\end{remark}

Let us continue with the detailed proof. 

\begin{proof}[Proof of Theorem \ref{functortheorem}]
Recall that $\dd A_i(M)= \dd (R \otimes \bigwedge^i W \otimes M)$ and $\dd B_i(M) = \dd ((R \otimes \bigwedge^i W)^{\vee} \otimes M)$. 
$r_M$ and $t_{N^*}$ are defined in \ref{r-iso} and \ref{t-iso} using certain resolutions. We resolve now $\tau^*$ and $\Ext^n(\tau)$ via these resolutions. The resolutions use the complex constructions Nielsen II (\ref{reslemma}) and Nielsen IIa (\ref{resIIa}).

At first we resolve $\tau$. Let $\tilde \tau_i : \dd(R \otimes \bigwedge^i W \otimes M) \rightarrow \dd((R \otimes \bigwedge^{n-i} W)^{\vee} \otimes N)$ be given by
\[
r \otimes w \otimes \mm \mapsto \left \{ \begin{array}{c} R\otimes \bigwedge^{n-i} W \rightarrow R \otimes \bigwedge^n W \\ (r' \otimes w') \mapsto (r r') \otimes w'\wedge w \end{array} \right \} \otimes \tau(\mm), 
\]
for all $\mm \in M$, $w \in \bigwedge^i W$, $w' \in \bigwedge^{n-i} W$ and $r ,r' \in R$ .  Moreover ${u_{N}}_i: P_i^{\vee}\otimes N \rightarrow P_i^{\vee}\otimes N^{**}$ is defined by $\id \otimes u_N.$ 
As an abbreviation write $v:= u_N \circ \tau$ and $v_i := {u_{N}}_i \circ \tilde \tau_{n-i}$. Let $\ell(i) =\lfloor \frac{i-1}{2} \rfloor$, then we obtain the following commutative diagram. 
\small
\[
\begin{array}{ccc}
0  \leftarrow M \leftarrow \dd A_0(M) & \ldots \leftarrow  \dd A_m(M)  \quad \stackrel{\dd\phi_{m+1}}{\leftarrow} \quad \dd A_{m+1}(M)   \ldots & \leftarrow  \dd A_n(M)  \leftarrow 0 \\
  \downarrow^{v} \qquad   \downarrow^{v_n}& \qquad\qquad \quad \downarrow^{(-1)^{\ell(m)} v_{m+1}} \quad  \qquad\qquad \downarrow^{(-1)^{(m+\ell(m))} v_m}   & \quad \downarrow^{(-1)^m v_0}\\
0  \leftarrow N^{**} \leftarrow \dd B_n(N^{**})& \ldots  \leftarrow \dd B_{m+1}(N^{**})  \stackrel{ \dd\varphi_{m+1}}{\longleftarrow}  \dd B_m(N^{**})  \ldots & \leftarrow   \dd B_0(N^{**}) \leftarrow 0, \\
\end{array}
\]  
\normalsize
where The diagram commutes by Lemma \ref{kosI}, because the $\tilde \tau_i$ behave exactely as the map between the Koszul complex and its dual.
\\
Applying $(\quad)^{\vee} = \Hom_R(\quad,R\otimes  \bigwedge^n W)$ to the diagram we obtain
\small
\[
\begin{array}{l}
0  \leftarrow \Ext^n(M) \leftarrow (\dd A_n(M))^{\vee} \ldots   (\dd A_{m+1}(M))^{\vee} \quad \stackrel{\dd\phi_{m+1}^{\vee}}{\leftarrow}  (\dd A_m(M))^{\vee} \qquad  \ldots(\dd A_0(M))^{\vee} \leftarrow 0\\
 \qquad\quad \uparrow^{ \Ext^n(v) }  \qquad   \uparrow^{{(-1)^m (v_0)^{\vee}}} \qquad \quad   \uparrow^{(-1)^{(m+\ell(m))} (v_{m+1})^{\vee}}  \qquad  \uparrow^{(-1)^{\ell(m)} (v_{m+1})^{\vee}}  \qquad  \quad\uparrow^{  (v_n)^{\vee}}  \\
0  \leftarrow \Ext^n(N^{**}) \leftarrow (\dd B_0(N^{**}) )^{\vee} \ldots   (\dd B_m(N^{**}))^{\vee} \stackrel{ \dd\varphi_{m+1}^{\vee}}{\leftarrow}  (\dd(B_{m+1}(N^{**}))^{\vee}  \ldots (\dd B_n(N^{**}))^{\vee} \leftarrow 0. \\
\end{array}
\]
\normalsize
\\
In the same manner as above we resolve $(-1)^m \tau^*$ (also using Nielsen II and Nielsen IIa):

\small
\[
\begin{array}{ccc}
0  \leftarrow M^*  \quad \leftarrow \dd B_n(M^*)  & \ldots \leftarrow \dd B_{m+1}(M^*) \stackrel{ \dd (\varphi_{m+1})}{\leftarrow} \dd B_{m}(M^*)   \ldots&\leftarrow   \dd B_0(M^*) \leftarrow 0 \\
\qquad   \uparrow^{(-1)^m \tau^*}   \qquad  \uparrow^{(-1)^m\widetilde {\tau_0^*}} &  \qquad \qquad \uparrow^{(-1)^{m+\ell(m)}\widetilde{\tau_m^*}}  \qquad \qquad \uparrow^{(-1)^{\ell(m)}\widetilde{\tau_{m+1}^*}} &\qquad  \uparrow^{\widetilde{ \tau_n^*}}\\
0  \leftarrow N^* \quad \leftarrow \dd A_0(N^*)  &\ldots \leftarrow  \dd A_m(N^*)   \stackrel{\dd (\phi_{m+1})}{\leftarrow}  \dd A_{m+1}(N^*)  \ldots & \leftarrow  \dd A_n(N^*) \leftarrow 0, \\
\end{array}
\]
\normalsize
where $\widetilde{ \tau_i^*}: \dd A_i(N^*) \rightarrow \dd B_{n-i}(M^*)$ is defined by  
\[ (r \otimes w \otimes \dn) \mapsto \left \{ \begin{array}{c} R\otimes \bigwedge^{n-i} W \rightarrow R \otimes \bigwedge^n W \\ r' \otimes w' \mapsto (r r') \otimes  w' \wedge w \end{array}  \right \} \otimes \tau^*(\dn) \] for all $\dn \in N^*$, $w \in \bigwedge^i W$, $w' \in \bigwedge^{n-i} W$ and $r ,r' \in R$ . 

\medskip

Consider the canonical isomorphism from \ref{caniso}: $\alpha_{M,i} : \dd B_i(M^*) \stackrel{\cong}{\rightarrow} (\dd A_i(M))^{\vee}$. Hence $(\alpha_{N^*,n-i})^{\vee}: \dd A_{n-i}(N^*)\stackrel{\cong}{\rightarrow} (\dd A_{n-i}(N^{*}))^{\vee})^{\vee} \stackrel{\cong}{\rightarrow} (\dd B_{n-i}(N^{**}))^{\vee}$. We  connect the two diagrams using them. Let us show the commutativity of the following diagram for all $n \geq i \geq 0$:

\vspace*{0.8cm}

\xymatrix{&&&&& \dd B_i(M^*)    \ar[rr]^{\alpha_{M,i}  \,\,\,\,}& & (\dd A_i(M))^{\vee} \\
          &&&&& \dd A_{n-i}(N^*)  \ar[u]^{\widetilde{\tau_{n-i}^*} \,\,}  \ar[rr]^{(\alpha_{N^*,n-i})^{\vee}}& &\quad (\dd B_{n-i}(N^{**}))^{\vee}\ar[u]_{(\tilde \tau_{i})^{\vee} \circ (u_{N,n-i})^{\vee}}. }

\medskip

Let $r \in R, \dn \in N^*, \mn^{**} \in N^{**}, w \in \bigwedge^{n-i} W, \tilde w \in \bigwedge^i W$ and $\mm \in M$ be arbitrary. Moreover let $\dmp \in P_{n-i}^{\vee}$, such that $1\otimes w \mapsto 1 \otimes w \wedge w'$ for some fixed $w' \in \bigwedge^i W$. 
\begin{eqnarray*}
\dd A_{n-i}(N^*) \ni 1\otimes w \otimes \dn \mapsto  \left \{\begin{array}{c} B_{n-i}(N^{**}) \rightarrow R\otimes \bigwedge^n W \\  \dmp \otimes  \mn^{**}  \mapsto  \dmp(w) \cdot \dn( \mn) \end{array} \right \} \stackrel{\tilde \tau_i^{\vee}\circ u^{\vee}_{N,n-i}}{\longmapsto} \\
\left \{ \begin{array}{c} A_i(M) \rightarrow R \otimes \bigwedge^n W \\ 
(1 \otimes \tilde w \otimes  \mm ) \mapsto (w \wedge \tilde w) \cdot \dn(\tau( \mm)) =
 (-1)^{i(n-i)} (\tilde w \wedge w) \cdot \dn(\tau( \mm)) \end{array} \right \},
\end{eqnarray*}
and first applying $\widetilde{\tau_{n-i}^*}$ we have
\begin{eqnarray*}
 1\otimes w \otimes \dn \mapsto (1 \otimes (\_\wedge w )\otimes \dn\circ\tau) \mapsto 
\left \{ \begin{array}{c} A_i(M) \rightarrow R\otimes \bigwedge^n W \\ (1\otimes \tilde w \otimes \mm ) \mapsto (\tilde w\wedge  w) \cdot \dn(\tau( \mm)) \end{array} \right \} ,\\
\end{eqnarray*}
which equals to the above expression as $n$ is odd. Note again that in the last row $1\otimes (\_\wedge w)$ stands for the functional $ R \otimes \bigwedge^i W \rightarrow R \otimes \bigwedge^n W, \quad r \otimes \tilde w \mapsto r \otimes(\tilde w \wedge w).$

\medskip

That means the right part of the next diagram commutes. On the other hand we derive of the equivalences of functors the diagram of the theorem on the left. For clarity we use a reduced notation:

\vspace{0.5 cm}

\begin{xy}
\xymatrix@!0{
&&&&&0 \leftarrow  \qquad \qquad & \Ext^n(M)  \quad\ar@{<-}'[d][dd]_{\Ext^n(\tau)} && \quad \leftarrow (A_n(M))^{\vee} \ar@{<-}'[d][dd] &\quad \quad\quad \quad \quad \quad  \leftarrow \ldots & &&&&\leftarrow (A_0(M))^{\vee}  \ar@{<-}'[d][dd] \leftarrow 0 \\
&&&&0 \leftarrow \qquad&M^* \ar@<1.2ex>[ur]_{r_M}  \ar@{<-}[dd]_{(-1)^m \tau^*} &&\longleftarrow B_n(M^*)  \ar@<1.2ex>[ur]_{\alpha_{M,n}} \ar@{<-}[dd]&\quad\quad \quad \quad \quad\quad\leftarrow \ldots&&&&&\longleftarrow B_0(M^*)  \ar@<1.2ex>[ur]_{\alpha_{M,0}} \ar@{<-}[dd]\leftarrow 0& \\
&&&&&& \Ext^n(N^{**}) && \quad \quad \leftarrow (B_0(N^{**}) )^{\vee}&\quad \quad \quad \quad\quad \quad \quad \leftarrow \ldots & \qquad \qquad \qquad \qquad \qquad&&&&\leftarrow (B_n(N^{**}) )^{\vee}\leftarrow 0&\\
&&&&0 \leftarrow \qquad &N^* \ar@{<-}[rr] \ar@<1.2ex>[ur]_{ t_{N^*}}&& A_0(N^*)  \ar@<1.2ex>[ur]_{(\alpha_{N^*,0})^{\vee}}&\qquad \qquad \leftarrow \ldots&&&&&\leftarrow A_n(N^*)  \ar@<1.2ex>[ur]_{(\alpha_{N^*,n})^{\vee}}\leftarrow 0&
}
\end{xy}

\end{proof}

Now we are able to prove the central theorem of the section, Theorem \ref{backwardstheorem}. It basically says that our definition of Gorensteiness is actually equivalent to having a selfdual resolution.

\begin{proof}[Proof of Theorem \ref{backwardstheorem}]
Let again $(\quad)^{\vee}=\Hom_R(\quad,R(-d))$. Using the given selfdual resolution, we define an isomorphism $\tau' : M \rightarrow \Ext_R^n(M,R(-d))(-s)$ as follows: We consider the dual of the resolution and obtain an obvious map of complexes as seen in the diagram. By abuse of notation, we denote by $\id: F_i \stackrel{\cong}{\mapsto} F_i^{\vee \vee}$ the canonical isomorphism, too.
\small
\[
\begin{array}{cccc}
 0 \leftarrow M &\leftarrow F_0 \stackrel{\psi_1}{\leftarrow}\quad \ldots \leftarrow F_{m}&\stackrel{\psi_{m+1}}{\leftarrow}  F^{\vee}_{m}(-s) &\leftarrow \ldots \stackrel{\psi_1^{\vee}(-s)}{\leftarrow} F_0^{\vee}(-s)\leftarrow 0\\
 \qquad \qquad  \downarrow^{\tau'} \quad &  \quad \downarrow^{id} \quad \qquad \qquad \downarrow^{id} &\quad \downarrow^{\pm id} & \qquad \quad \downarrow^{\pm id}   \\
 0 \leftarrow \Ext^n_R(M,R(-d-s)) &\leftarrow {F_0^{\vee}}^{\vee}  \leftarrow \ldots \leftarrow {F_m^{\vee}}^{\vee}   &\stackrel{\pm\psi_{m+1}}{\leftarrow}  F^{\vee}_{m}(-s) &\leftarrow \ldots \stackrel{\psi_1^{\vee}(-s)}{\leftarrow} F_0^{\vee}(-s)\leftarrow 0.
\end{array}
\] 
\normalsize
\\
Applying $\Hom_R(\quad,R(-d))$ to the diagram, we obtain (here $\Ext^n_R(\quad)$ denotes $\Ext^n_R(\quad,R(-d))$:
\small
\[
\begin{array}{ccc}
 0 \leftarrow \Ext^n_R(M) \qquad \leftarrow (F_0^{\vee}(-s))^{\vee} \quad \ldots    (F_m^{\vee}(-s))^{\vee}\quad \stackrel{\pm\psi_{m+1}}{\leftarrow}  F^{\vee}_{m} &\ldots  \leftarrow F_0^{\vee} \leftarrow 0\\
 \quad \qquad  \uparrow^{\Ext^n_R(\tau')} \qquad  \quad \quad \uparrow^{\pm id} \qquad\qquad \qquad\uparrow^{\pm id} \quad \qquad \qquad \uparrow^{ id} & \qquad  \uparrow^{ id}\quad\\
 0 \leftarrow \Ext^n_R(\Ext^n_R(M)(-s)) \quad \leftarrow (F_0^{\vee}(-s))^{\vee}   \ldots (F_m^{\vee}(-s))^{\vee} \quad \stackrel{\psi_{m+1}}{\leftarrow}  ((F_m^{\vee})^{\vee})^{\vee}  &\ldots  \leftarrow((F_0^{\vee})^{\vee})^{\vee} \leftarrow 0.
\end{array}
\] 
\normalsize
We want to make use of the machinery developed within this section. Especially our aim is to apply Theorem \ref{functortheorem}. Therefore in what follows we identify $R \otimes \bigwedge^n W \cong R(-d)$, via $r \otimes \chi_1 \wedge \ldots \wedge \chi_n \mapsto r.$
\\
From the diagrams we obtain (compare Definition \ref{s-iso})
\[
\Ext_R^n(\tau')(-s) = \pm \tau' \circ s_M^{-1}. \mbox{   (*)}  
\]
Finally we use the isomorphism $r_M: M^* \stackrel{\cong}{\rightarrow} \Ext^n_R(M,R(-d))$ (see \ref{r-iso}) to define $\tau : M \rightarrow M^*(-s)$ as \[\tau := r_M^{-1}(-s)  \circ \tau'. \mbox{   (**)} \] 

We apply Theorem \ref{functortheorem} to the situation $N = M^*(-s)$.  Hence we obtain (note that by definition $t_N = \Ext^n(r_N) \circ s_N$ (see \ref{t-iso})): 
\begin{eqnarray*}
\tau^* = (-1)^{m} r_M^{-1} \circ \Ext^n(\tau) \circ \Ext^n(u_{M^*})(s)  \circ t_{M^{**}}(s) \stackrel{(**)}{=} \\
(-1)^{m} r_M^{-1} \circ \Ext^n(\tau') \circ \Ext^n(r_M)^{-1}(s)\circ \Ext^n(u_{M^*})(s)   \circ \Ext^n(r_{M^{**}})(s) \circ s_{M^{**}}(s)\stackrel{(*)}{=}\\
\pm (-1)^{m} r_M^{-1} \circ \tau'(s) \circ s_M^{-1} (s) \circ \Ext^n(r_M)^{-1}(s) \circ \Ext^n(u_{M^*})(s)  \circ \Ext^n(r_{M^{**}})(s) \circ s_{M^{**}}(s) \stackrel{(1)}{=}\\
\pm (-1)^{m} r_M^{-1} \circ \tau'(s)\circ u_M^{-1}(s) \stackrel{(**)}{=}\\
\pm (-1)^{m} \tau(s) \circ u_M^{-1}(s).
\end{eqnarray*}

For $(1)$ it remains to show the commutativity of the following diagram:

\medskip

\small
\[
\begin{array}{ccccc}
  M^{**} &\stackrel{s_{M^{**}}}{\longrightarrow} & \Ext^n(\Ext^n(M^{**}))&\stackrel{\Ext^n(r_{M^{**}})}{\longrightarrow}  & \Ext^n(M^{***})\\
 \downarrow^{u_{M}^{-1}} & (2) &\quad \downarrow^{\Ext^n(\Ext^n(u_M)^{-1})} & (3) & \downarrow^{\Ext^n(u_{M^*})} \\
 M & \stackrel{s_{M}^{-1}}{\longleftarrow} & \Ext^n(\Ext^n(M)) & \stackrel{\Ext^n(r_M)^{-1}}{\longleftarrow} & \Ext^n(M^*). 
 \end{array}
 \] 
\normalsize
Diagram $(2)$ commutes as $\{ M \mapsto s_M \,|\, M \in \Obj(\grMFL)\}$ is an isomorphism of functors (see \ref{s-iso}): $\id \rightarrow \Ext^n(\Ext^n(\quad))$. 
\\  
Diagram$(3)$ commutes as 

\medskip

\[
\xymatrix{& M^{***}  \ar[r]^{\,\,\quad u_{M^*}^{-1}  } \ar[d]^{r_{M^{**}}} & M^*\ar[d]^{r_M}\\
          & \Ext^n(M^{**}) \ar[r]^{\, \Ext^n(u_M) }& \Ext^n(M) }  \mbox{    (3a)}
\]
commutes by using the isomorphism of functors property of $\{ M \mapsto r_M \,|\, M \in \Obj(\grMFL)\}$, $(\quad)^{*} \rightarrow \Ext^n(\quad)$. Apply $\Ext^n(\quad)$ to $(3a)$ and gain
\[
\xymatrix{& \Ext^n( M^{***} )&& \ar[ll]_{\,\,\quad \Ext^n( u_{M^*})^{-1}  } \Ext^n( M^*)\\
          & \Ext^n(\Ext^n(M^{**})) \ar[u]^{\Ext^n(r_{M^{**}})} && \ar[ll]_{\, \Ext^n(\Ext^n(u_M)) }\Ext^n(\Ext^n(M)) \ar[u]^{\Ext^n(r_M)} .}
\]
Inverting all maps on the right hand side gives $(3)$. Hence we have finished the proof. 
\end{proof}

\begin{remark}
In the proof --- omitting for a second the technical details --- we see the reason why we needed $t_{M^*}$ to be mainly $\Ext^n(r_{M^*})$: It vanishes together with $\Ext^n(r_M)^{-1}$ form the definition of $\Ext^n(\tau)$, and we can compute $\tau^*$ in terms of $\tau$. 
\end{remark}

\end{document}